\documentclass[twocolumn]{svjour3}
\smartqed  
\usepackage{amsmath}
\usepackage{amsfonts}
\usepackage[utf8]{inputenc}
\usepackage[T1]{fontenc}
\usepackage{graphicx}
\usepackage{ifpdf}
\ifpdf
\usepackage[all,pdf,2cell]{xy}\UseAllTwocells\SilentMatrices
\else
\usepackage[all,xdvi,2cell]{xy}\UseAllTwocells\SilentMatrices
\fi
\newcommand{\GEA}{{\bf GEA}}
\newcommand{\EA}{{\bf EA}}
\newcommand{\Fu}{F}
\newcommand{\Uu}{U}
\newcommand{\F}[1]{\Fu(#1)}
\newcommand{\U}[1]{\Uu(#1)}
\newcommand{\FUu}{\Fu\Uu}
\newcommand{\UFu}{\Uu\Fu}
\newcommand{\FUFu}{\Fu\Uu\Fu}
\newcommand{\UFUu}{\Uu\Fu\Uu}
\newcommand{\FU}[1]{\FUu(#1)}
\newcommand{\UF}[1]{\UFu(#1)}
\newcommand*{\vcenteredhbox}[1]{\begingroup
\setbox0=\hbox{#1}\parbox{\wd0}{\box0}\endgroup}

\newcommand{\adj}[1][]{\def\ArgI{#1}\adjRelayI}
\newcommand{\adjRelayI}[1][]{\def\ArgII{#1}\adjRelayII}
\newcommand{\adjRelayII}[3][2.2em]{\ensuremath{\SelectTips{cm}{10}\xymatrix@C=#1@1{{#2} \ar@<1ex>[r]^-{\ArgI}^-{}="1" & {#3} \ar@<1ex>[l]^-{\ArgII}^-{}="2" \ar@{}"1";"2"|(.3){\hbox{}}="7" \ar@{}"1";"2"|(.7){\hbox{}}="8" \ar@{|-} "8" ;"7"}}}

\newcommand{\radjRelayII}[3][2.2em]{\ensuremath{\SelectTips{cm}{10}\xymatrix@C=#1@1{{#2} \ar@<-1ex>[r]_-{\ArgI}^-{}="1" & {#3} \ar@<-1ex>[l]_-{\ArgII}^-{}="2" \ar@{}"1";"2"|(.3){\hbox{}}="7" \ar@{}"1";"2"|(.7){\hbox{}}="8" \ar@{|-} "7" ;"8"}}}
\begin{document}

\title{A note on unitizations of generalized effect algebras}
\author{Gejza Jenča}
\institute{G. Jenča \at
Department of Mathematics and Descriptive Geometry\\
Faculty of Civil Engineering\\
Slovak University of Technology\\
Radlinsk\' eho 11\\
	Bratislava 810 05\\
	Slovak Republic\\
              \email{gejza.jenca@stuba.sk}           
}
\maketitle
\begin{abstract}
There is a forgetful functor from the category of generalized
effect algebras to the category of effect algebras. We prove that
this functor is a right adjoint and that the corresponding left
adjoint is the well-known unitization construction by Hedlíková
and Pulmannová. Moreover, this adjunction is monadic.
\end{abstract}
\section{Introduction and preliminaries}

\subsection{Introduction}

In \cite{hedlikova1996generalized}, authors proved
that every generalized effect algebra can be embedded into an effect algebra.
The construction was subsequently studied and applied by several authors,
for example in \cite{paseka2009isomorphism}, \cite{riecanova2005generalized},
\cite{riecanova2008effect}. A generalization of the unitization construction
to pseudoeffect algebras was recently introduced and 
studied in \cite{foulis2014unitizing}

It is easy to see that this unitization construction is functorial. 
We prove that this {\em unitization functor} is left adjoint to the
forgetful functor from the generalized effect algebras to effect algebras and
that this adjunction is monadic. Thus, the category of effect algebras
is the category of algebras for a monad defined on the category of generalized
effect algebras.

We assume working knowledge of basic category theory \cite{mac1998categories} 
and theory of effect algebras \cite{DvuPul:NTiQS}.

\subsection{Generalized effect algebras}

Let $P$ be a partial algebra with a nullary operation $0$
and a binary partial operation $\oplus$. Denote the domain of $\oplus$ by
$\perp$. $P$ is called a {\em generalized effect algebra}
iff for all $a,b,c\in P$ the following conditions are satisfied :
\begin{enumerate}
\item[(P1)] $a\perp b$ implies $b\perp a$, $a\oplus b=b \oplus a$.
\item[(P2)] $b\perp c$ and $a\perp b\oplus c$ implies
        $a\perp b$, $a\oplus b\perp c$, $a\oplus (b\oplus c)=(a\oplus
        b)\oplus c$.
\item[(P3)] $a\perp 0$ and $a\oplus 0=a$.
\item[(P4)] $a\oplus b=a\oplus c$ implies $b=c$.
\item[(P5)] $a\oplus b=0$ implies $a=0$.
\end{enumerate}

In a generalized effect algebra $P$, we denote $a\leq b$ iff
$a\oplus c=b$ for some $c\in P$. It is easy to see
that $\leq$ is a partial order and that $0$ is the least element of
the poset $(P,\leq)$.
We denote $c=a\ominus b$ iff $a=b\oplus c$. Owing to cancellativity,
$\ominus$ is a well-defined partial operation with domain $\geq$.

Let $P_1$, $P_2$ be generalized effect algebras. A map $f:P_1\to P_2$ is called a
{\em morphism of generalized effect algebras} if and only if it satisfies the following conditions.
\begin{itemize}
\item $f(0)=0$.
\item If $a\perp b$, then $f(a)\perp f(b)$ and $f(a\oplus b)=f(a)\oplus f(b)$.
\end{itemize}

A morphism is $f:P_1\to P_2$ is {\em full} if $f(a)\perp f(b)$ implies that there
are $a_1,b_1\in P_1$ such that $a_1\perp b_1$, $f(a)=f(a_1)$ and $f(b)=f(b_1)$.

\subsection{Effect algebras}

An {\em effect algebra} is an generalized effect algebra bounded above.
Unwinding this definition, we observe that an effect algebra is partial algebra $(E;\oplus,0,1)$ with a binary
partial operation $\oplus$ and two nullary operations $0,1$ such that the
reduct $(E;\oplus,0)$ is a generalized effect algebra and $1$ is the greatest
element of $E$.

Effect algebras were introduced by Foulis and Bennett in their paper 
\cite{FouBen:EAaUQL}. See also \cite{KopCho:DP} and \cite{GiuGre:TaFLfUP} for
equivalent definitions, introduced independently.

Let $E_1$, $E_2$ be effect algebras. A map $\phi:E_1\to E_2$ is called a
{\em morphism of effect algebras} if and only if it is a morphism of 
generalized effect algebras satisfying the condition $\phi(1)=1$. 
A morphism $\phi:E_1\to E_2$ is a {\em full morphism} if and only if
$\phi(a)\perp\phi(b)$ implies that there are $x,y\in E_1$ such that
$\phi(x)=a$, $\phi(y)=b$ and $x\perp y$. A full, bijective morphism is an {\em isomorphism}.

\section{The unitization functor}

The category of generalized effect algebras is denoted by $\GEA$,
the category of effect algebras is denoted by $\EA$,
$\Uu:\EA\to\GEA$ is the evident forgetful functor.

Let us define a functor $F:\GEA\to\EA$, called {\em unitization}.

For a generalized effect algebra $P$, $\F P$ is a partial algebra with an underlying set
$P\dot\cup P^*$, where $P\cap P^*=\emptyset$ and $a\mapsto a^*$ is a
bijection from $P$ to $P^*$, equipped with a partial
operation given as follows: for all $a,b\in P$,
\begin{itemize}
\item $a\perp_{\F P} b$ iff $a\perp_P b$ and then $a\oplus_{\F P} b=
a\oplus_P b$,
\item $a\perp_{\F P} b^*$ iff $a\leq b$ and then $a\oplus_{\F P} b^*=
(b\ominus_P a)^*$,
\item $a^*\perp_{\F P} b$ iff $a\geq b$ and then $a^*\oplus_{\F P} b=
(a\ominus_P b)^*$,
\item $a^*\not\perp_{\F P} b^*$.
\end{itemize}
This construction was introduced by Hedlíková and Pulmannová in \cite{hedlikova1996generalized}.
They proved that $(F(P),\oplus,0,0^*)$ is always
an effect algebra. The basic idea of the construction predates effect algebras, see
\cite{janowitz1968note}, \cite{mayet1991generalized}. 
\begin{figure}
\vcenteredhbox{
	\includegraphics{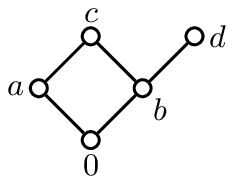}
}
\vcenteredhbox{
	\includegraphics{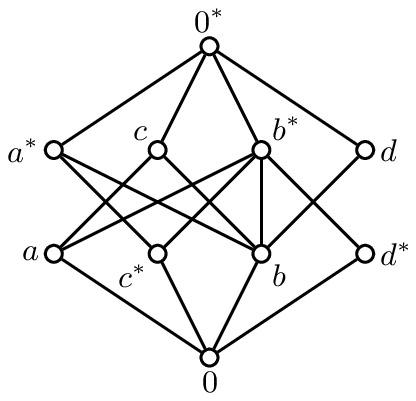}
}
\caption{A generalized effect algebra and its unitization}
\label{fig}
\end{figure}
\begin{example}
Let $P$ be the poset on the left-hand side of Figure~\ref{fig}.
There is a unique $\oplus$ partial operation on $P$, making
$(P,\oplus,0)$ into a generalized effect algebra: $a\oplus b=b\oplus a=c$, $b\oplus b=d$
and $0\oplus x=x\oplus 0=x$, for all $x\in P$.

The Hasse diagram of the effect algebra $F(P)$ appears on the right-hand side
of the picture.
\end{example}
\begin{example}
Let us consider generalized effect algebra (in fact, a total monoid)
$(\mathbb{N},+,0)$, where $+$ is the ordinary addition of natural numbers.
Then $F(\mathbb{N})$ is a totally ordered MV-algebra, also 
known under the name {\em Chang's MV-algebra}.
\end{example}

For a morphism of generalized effect algebras $f:P\to Q$, then $\F f:\F P\to\F Q$ is given by
$\F f(a)=a$, $\F f(a^*)=(f(a))^*$. It is easy to check that $\F f$ is a
morphism in the category $\EA$ and that $\Fu:\GEA\to\EA$ is a functor.

\begin{theorem}
$\Fu$ is left adjoint to $\Uu$.
\end{theorem}
\begin{proof}
Let  us define the unit $\eta$. For every generalized effect algebra $P$, the
component $\eta_P:P\to \UF P)$ is the embedding $\eta_P(x)=x$. This is
obviously a natural transformation $1_{\GEA}\to \UFu$.

Let $E$ be an effect algebra; to define the component of the counit $\epsilon$
at $E$, we need to take a closer look at $\FU E$. Let us prove that $w:\FU E\to
E\times\{0,1\}$ given by $w(a)=(a,0)$ and $w(a^*)=(a',1)$, where $a\in E$, 
is an isomorphism of
effect algebras. Indeed, suppose that $x,y\in\FU E$ and that $x\perp y$.  The
only nontrivial case we have to check is when there  are $a,b\in E$ such that
$x=a$, $y=b^*$ and $a\leq b$; in this case $x\oplus y=(b\ominus_E a)^*$ and 
\begin{multline*}
w(x\oplus y)=w((b\ominus_E a)^*)=((b\ominus a)',1)\\=(b'\oplus a,1)=(a,0)\oplus(b',1)=
w(a)\oplus w(b^*)=w(x)\oplus w(y).
\end{multline*}

The morphism $w$ is easily seen to be full and bijective, hence an isomorphism.  

We may now define $\epsilon_E:\FU E\to E$ as
the composition of $w$ with the canonical projection $p:E\times\{0,1\}\to E$.
Explicitly, $\epsilon_E(x)=x$ for $x\in E$ and $\epsilon_E(x^*)=x'$ for
$x^*\in E^*$.  The commutativity of the naturality square of $\epsilon$ also
clear.

Let us check the triangle identities. We need to prove that, in the categories
of endofunctors of $\GEA$  and $\EA$, respectively, the triangles
$$
\xymatrix{
\Fu\ar[r]^-{\Fu\eta}\ar[rd]_{1_\Fu}&\FUFu\ar[d]^{\epsilon\Fu}\\
 &\Fu
}
\quad
\xymatrix{
\Uu\ar[r]^-{\eta\Uu}\ar[rd]_{1_\Uu}&\UFUu\ar[d]^{\Uu\epsilon}\\
 &\Uu
 }
$$
commute. 

To observe the commutativity of the left triangle, let $P$ be a generalized
effect algebra. If $x\in P$, then $\F{\eta_P}(x)=x$ and $\epsilon_{\F P}(x)=x$.
If $x^*\in P$, then $\F{\eta_P(x^*)}=(\eta_P(x))^*=x^*$ and $\epsilon_{\F
P}(x^*)=x^*$.

To observe the commutativity of the right triangle, let $E$ be an effect
algebra and let $x\in\U E$. 
Then $\eta_{\U E}(x)=x$ and $\U{\epsilon_{\FUu}(x)}=x$.  
\end{proof} 

Let us consider the real interval $[0,1]_{\mathbb R}$, equipped with the usual addition of
real numbers restricted to $[0,1]_{\mathbb R}$, meaning that $a\oplus b$ is defined
if and only if $a+b\leq 1$ and then $a\oplus b:=a+b$. A morphism of
generalized effect algebras $P\to[0,1]_{\mathbb R}$ is called an {\em additive map on $P$}.

For an effect algebra $E$, a {\em state on $E$} is an additive map $E\to [0,1]_{\mathbb R}$
preserving the unit, so a state is a morphism in $\EA$.

\begin{corollary}
\label{cor:stateextends}
Every additive map on $P$ uniquely extends to a state on $F(P)$.
\end{corollary}
\begin{proof}
If $s$ is an additive map, then there is a unique $\bar s:F(P)\to [0,1]_{\mathbb R}$
such that the diagram
$$
\xymatrix{
UF(P)\ar[r]^{U(\bar s)}&U([0,1]_{\mathbb R})\\
P\ar[u]_{\eta_P}\ar[ur]_f
}
$$
commutes.
\end{proof}

Every state $f$ on an effect algebra must satisfy $f(x')=1-f(x)$. Therefore,
if $s$ is an additive map on $P$, then the state $\bar s$ on $F(P)$ corresponding
to $s$ via the bijection established in \ref{cor:stateextends} is necessarily given by
$$
\bar s(x)=
\begin{cases}
s(x)& x\in P\\
1-s(x)& x\in P^*
\end{cases}
$$

In a very similar way, one can prove the following:
\begin{corollary}
There is a natural one-to-one correspondence
between ideals of $P$ and morphisms $F(P)\to 2^2$, where $2^2$ is the Boolean algebra with two atoms.
\end{corollary}

\begin{lemma}
\label{lemma:coequalizers}
The forgetful functor $U:\EA\to\GEA$ creates coequalizers.
\end{lemma}
\begin{proof}
Let $f,g\colon A\to B$ be a pair of morphisms in $\EA$, let
$h\colon B\to Z$ be a coequalizer of $f,g$ in $\GEA$. We need to prove that
$h(1)$ is the top element of $Z$.
Consider the diagram
$$
\xymatrix{
U(A)\ar@<0.7ex>[r]^{U(f)}\ar@<-0.7ex>[r]_{U(g)}&U(B)\ar[r]^h\ar[rd]_t&Z\ar@<0.7ex>@{.>}[d]^u\\
~&~&{[0,h(1)]_Z}\ar@<0.7ex>@{^{(}->}[u]^j
}
$$
where $t$ is given by $t(x)=h(x)$ and $j$ is the inclusion into $Z$, so that
$j\circ t=h$.  Since $t\circ f=t\circ g$, there is a unique $u\colon
Z\to[0,h(1)]_Z$ such that $u\circ h=t$.  This gives us $j\circ u\circ h=h$.
Since $h$ is an epimorphism, $j\circ u=id_Z$ and now we see that for every $x\in
Z$, $x=j(u(x))\leq h(1)$, because the range of $j$ is bounded above by $h(1)$.

It remains to prove that $h$ is a coequalizer in $\EA$. If $E$ is a
generalized effect algebra bounded above and $s$ is a top-preserving morphism such that
$s\circ f=s\circ g$, then there is a unique $\GEA$-morphism $u$ such that $u\circ h=s$.
However, since both $h$ and $s$ preserve the top element, $u$ must be top-preserving
as well. 
\end{proof}
Recall \cite{mac1998categories}, that every adjunction 
$$
\adj[F][U]{\mathcal{C}}{\mathcal{D}}
$$ gives rise to a monad
$(UF,\eta,U\epsilon F)$ on $\mathcal{C}$. An adjunction is {\em monadic} if $U$ is equivalent to
the forgetful functor coming from the category of algebras $\mathcal{C}^{UF}$ 
for the monad $(UF,\eta,\epsilon)$ and the comparison gives us then an isomorphism 
$\mathcal{D}\simeq\mathcal {C}^{UF}$.
\begin{theorem}
The adjunction $(F,U,\eta,\epsilon)$ is monadic.
\end{theorem}
\begin{proof}
By Beck's theorem \cite{mac1998categories}, an adjunction is monadic
if and only if $U$ creates absolute coequalizers. By Lemma \ref{lemma:coequalizers},
$U$ creates all coequalizers.
\end{proof}
\begin{corollary}
$\EA\simeq\GEA^{UF}$.
\end{corollary}
\begin{acknowledgements}
This research is supported by grant VEGA G-2/0059/12 of M\v S SR,
Slovakia and by the Slovak Research and Development Agency under the contracts
APVV-0073-10, APVV-0178-11.
\end{acknowledgements}

\end{document}